\documentclass{article}
\usepackage{xcolor} 
\usepackage{amsmath,amssymb, amsthm, arydshln, bm, enumitem}

\allowdisplaybreaks

\theoremstyle{plain}
\newtheorem{thm}{Theorem}

\newtheorem{lem}[thm]{Lemma}
\newtheorem{rem}[thm]{Remark}

\theoremstyle{definition}
\newtheorem{dfn}[thm]{Definition}
\newtheorem{ex}[thm]{Example}
\newtheorem{problem}[thm]{Problem}
\newtheorem{fc}[thm]{Final Conclusion}

\title{\bf
Infinitesimal Conformal Rigidity on Damek–Ricci Spaces}
\author{Hiroyasu Satoh and Hemangi Madhusudan Shah}
\date{\today}

\begin{document}

\maketitle

\begin{abstract}
We show that every conformal vector field on a Damek–Ricci space is necessarily Killing, establishing a strong form of infinitesimal conformal rigidity.
Although this rigidity phenomenon is classically known in the Einstein setting, our proof follows a completely different approach.
We formulate the conformal Killing condition as an explicit system of partial differential equations on the solvable Lie group model and analyze it directly.
This local and analytic method yields a constructive proof of rigidity without relying on global arguments or transformation groups.
\end{abstract}

\begin{itemize}[leftmargin=73pt]
    \item[{\bf Keywords:}] Damek--Ricci space; Conformal vector field;  Killing vector field; 
\end{itemize}


\section{Introduction}

Conformal vector fields are smooth vector fields $\xi$ on a Riemannian manifold $(M, g)$ whose local flows preserve the conformal structure of the metric,
that is, they satisfy $\mathcal{L}_\xi g = 2\rho g$ for some smooth function $\rho$,
called the potential function.
When $\rho \equiv 0$, the field $\xi$ is Killing, meaning it generates an isometry.
It is a classical fact that space forms --- Riemannian manifolds of constant sectional curvature --- admit the richest possible collection of conformal vector fields.
Indeed, the space of all conformal vector fields of an $n$-dimensional space has maximal dimension $\frac{(n+1)(n+2)}{2}$, and this bound characterizes conformally flat space (see \cite[p.164, Theorem 3.2]{Yano1957}).
Classical conformal–rigidity goes back to a conjecture of
Lichnerowicz, stating that every complete Riemannian manifold whose conformal group is essential must be conformally equivalent to either the Euclidean space or the standard sphere.
The compact case was confirmed by Lelong–Ferrand and Obata, and the non–compact case was 
established by Ferrand, completing and correcting an earlier approach of Alekseevski\u{i}; see \cite{Ferrand1996, Alekseevskii72}.
Beyond the global theory of conformal transformation groups, the existence of nontrivial conformal vector fields has also been studied from various geometric viewpoints.
In particular, in the setting of Einstein manifolds, classical results due to Tashiro \cite{Tashiro1965} and Kanai\cite{Ka} provide a strong classification of those admitting non-homothetic conformal vector fields, showing that geometry of such manifolds is constrained 
to a large extent.

In recent work \cite{HSHS}, 
the authors revisited the problem of existence of conformal
vector fields on complex hyperbolic spaces $\mathbb{C}H^n$ for $n \ge 2$, and showed by a direct analytic approach that every conformal vector field is necessarily Killing.
The proof is based on writing $\mathbb{C}H^n$ as a solvable Lie group with a left-invariant metric -- a special case of what is known as a Damek--Ricci space -- and analyzing the conformal Killing equation in this setting.
By decomposing the resulting overdetermined system of partial differential equations,
the authors exploited properties of harmonic and holomorphic functions to reduce the problem to the vanishing of a conformal factor.
This was achieved using expansions in terms of harmonic polynomials and their uniqueness properties.

The present paper generalizes that result to the setting of all Damek--Ricci spaces.
These spaces, introduced by Damek and Ricci in the 1990s \cite{DR}, are solvable Lie groups equipped with a canonical left-invariant metric, constructed from generalized Heisenberg algebras.
They were discovered as counterexamples to the Lichnerowicz conjecture,
which proposed that all harmonic manifolds (i.e., manifolds for which geodesic spheres have constant mean curvature) must be locally symmetric.
Damek--Ricci spaces are non-symmetric harmonic manifolds of nonpositive curvature,
and include 
the noncompact rank-one symmetric spaces with nontrivial center (complex, quaternionic and Cayley hyperbolic spaces). The real hyperbolic space corresponds to the degenerate case where the center vanishes; while not an H-type (hence not a Damek–Ricci space in the strict sense), we treat it separately when convenient.
Since every Damek–Ricci metric is Einstein with negative scalar curvature,
the nonexistence of non-homothetic conformal vector fields on such spaces is already implied at the level of existence by the classification results of Tashiro and Kanai for Einstein manifolds (see \cite[Thm.~G]{Ka}).
The contribution of the present paper is to provide a completely different perspective: a direct and local analysis of the conformal Killing equation on the solvable model, yielding a constructive PDE proof that avoids any global splitting or group–action arguments.


Our motivation stems from the observation that, for any Damek--Ricci space $S$, the Lie algebra of the complex hyperbolic plane $\mathbb{C}H^2$ naturally embeds into that of $S$ as a Lie subalgebra.
Geometrically, this means that every Damek--Ricci space contains a totally geodesic submanifold isometric to the complex hyperbolic plane.
This structural feature led us to ask whether the rigidity result on $\mathbb{C}H^n$ could be extended to the broader class of Damek--Ricci spaces.
We show that this is indeed the case, and we state our main result as follows:
\begin{thm}\label{mainDR}
Let $(M, g)$ be a Damek--Ricci space and $\xi$ a conformal vector field.
Then the potential function $\rho$ of $\xi$ must vanish identically.
That is, $\xi$ is a Killing vector field.
\end{thm}


The proof follows the general strategy developed in \cite{HSHS} for complex hyperbolic space,
with suitable adaptations to the Damek–Ricci setting.
More precisely, we reduce the problem to a system of differential equations on a four-dimensional subalgebra isomorphic to the complex hyperbolic plane, and show that the conformal factor must vanish identically.

This result places Damek--Ricci spaces in striking contrast to real hyperbolic space,
which admits many non-Killing conformal vector fields constructed from Busemann functions.
It reinforces the view that, despite their rich geometry,
Damek--Ricci spaces exhibit a strong form of conformal rigidity analogous to the classical theorem of Lichnerowicz for compact Einstein K\"ahler manifolds (see \cite[2.125 Corollary]{Besse}).




\medskip

This paper is organized as follows.  
In Section 2, we review basic definitions and properties of conformal vector fields and introduce the structure of Damek--Ricci spaces.
Section 3 formulates the conformal Killing equation in the context of general Damek--Ricci spaces and rewrites it using left-invariant vector fields.
In Section 4, we analyze part of this system, showing that certain components must correspond to the real and imaginary parts of a holomorphic function.
Section 5 completes the proof of the main theorem by showing that the potential function of any conformal vector field must vanish, and discusses a related open problem motivated by the analytic structure revealed in the proof.

\section{Preliminaries}

\subsection{Conformal vector fields}

In this section, we recall the basic theory of conformal vector fields on Riemannian manifolds,
and summarize several classical results about the structure of the conformal transformation group.
We also present the explicit expressions for conformal vector fields on space forms of constant sectional curvature, such as the Euclidean space $\mathbb{E}^n$, the sphere $\mathbb{S}^n$, and the real hyperbolic space $\mathbb{R}H^n$.
\medskip

Let $(M, g)$ be a complete Riemannian manifold and let $\mathfrak{X}(M)$ denote the Lie algebra of smooth vector fields on $M$.

\begin{dfn}
A smooth vector field $\xi$ is said to be a \emph{conformal} vector field if there exists a smooth function $\rho \in C^\infty(M)$ such that
\begin{equation}\label{CVF}
\mathcal{L}_{\xi}g = 2\rho g,
\end{equation}
where $\mathcal{L}_{\xi}g$ denotes the Lie derivative of $g$ with respect to $\xi$.
The function $\rho$ is called the \emph{potential function} of $\xi$.
In particular, if $\rho$ is constant, then $\xi$ is called a \emph{homothetic} vector field; if $\rho \equiv 0$, then $\xi$ is a \emph{Killing} vector field.
\end{dfn}

Conformal vector fields are infinitesimal generators of conformal transformations on $(M, g)$, that is, smooth one-parameter families of diffeomorphisms $\{\varphi_t\}$ such that $\varphi_t^*g = e^{2\rho_t}g$ for some smooth functions $\rho_t$ depending on $t$.
This classical correspondence between conformal vector fields and conformal transformations has been studied in detail in the foundational works of Yano \cite[Chapter VII]{Yano1957}.

\begin{lem}\label{property_L}
The Lie derivative of the metric $g$ satisfies the following properties:
\begin{enumerate}
\item $\mathcal{L}_{\xi_1+\xi_2}g = \mathcal{L}_{\xi_1}g + \mathcal{L}_{\xi_2}g$ \quad for all $\xi_1, \xi_2 \in \mathfrak{X}(M)$.
\item $\mathcal{L}_{f\xi}g = f\,\mathcal{L}_{\xi}g + 2\,\mathrm{sym}(df \otimes \xi^\flat)$ \quad for all $f \in C^\infty(M)$ and $\xi \in \mathfrak{X}(M)$.
\end{enumerate}
Here, $\xi^\flat = g(\xi,\,\cdot\,)$ denotes the metric dual of $\xi$, and
\[
\mathrm{sym}(df \otimes \xi^\flat) := \dfrac{1}{2}(df \otimes \xi^\flat + \xi^\flat \otimes df).
\]
\end{lem}

\begin{proof}
We prove only (2), since (1) follows from the linearity of the Lie derivative.
Let $X, Y \in \mathfrak{X}(M)$ be arbitrary vector fields. Then
\begin{align*}
(\mathcal{L}_{f\xi}g)(X, Y)
&= g(\nabla_X(f\xi), Y) + g(X, \nabla_Y(f\xi)) \\
&= X(f)\,g(\xi, Y) + f\,g(\nabla_X \xi, Y) + Y(f)\,g(X, \xi) + f\,g(X, \nabla_Y \xi) \\
&= f\,\mathcal{L}_{\xi}g(X, Y) + 2\,\mathrm{sym}(df \otimes \xi^\flat)(X, Y),
\end{align*}
as claimed.
\end{proof}

It is known that the space of conformal vector fields forms a real vector space, and its dimension coincides with that of the local conformal transformation group.
A classical result (see \cite[p.164, Theorem 3.2]{Yano1957}) 
states that this dimension is bounded above by $\frac{1}{2}(n+1)(n+2)$ for an $n$-dimensional Riemannian manifold with $n \geq 3$.
Moreover, 
real space forms represent the maximal possible case for conformal symmetries.
We now illustrate the structure of conformal vector fields on simply connected space forms of constant sectional curvature—namely, the Euclidean space $\mathbb{E}^n$, the unit sphere $\mathbb{S}^n(1)\subset\mathbb{R}^{n+1}$, and the real hyperbolic space $\mathbb{R}H^n(-1)$ of constant curvature $(-1)$. 

\begin{ex}
In the $n$-dimensional Euclidean space $\mathbb{E}^n$, the conformal vector fields are generated by translations, rotations, homotheties, and special conformal transformations.
Explicitly, any conformal vector field $\xi$ on $\mathbb{E}^n$ can be written as
\[
\xi(\bm{x}) = \bm{a} + A\bm{x} + b_1 \bm{x} + \dfrac{1}{2}|\bm{x}|^2 \bm{b}_2 - \langle \bm{b}_2, \bm{x} \rangle \bm{x},
\]
where $\bm{a}, \bm{b}_2 \in \mathbb{R}^n$ are constant vectors, $A \in \mathfrak{so}(n)$ is a skew-symmetric matrix representing infinitesimal rotation, and $b_1 \in \mathbb{R}$ corresponds to infinitesimal scaling \cite{BL}.
The potential function is given by $\rho(\bm{x})=b_1 - \langle \bm{b}_2, \bm{x}\rangle$.
\end{ex}

\begin{ex}
Let $\mathbb{S}^n(1) \subset \mathbb{R}^{n+1}$ be the unit sphere in Euclidean space. For any constant vector $\bm{a} \in \mathbb{R}^{n+1}$ and skew-symmetric $(n+1)-$matrix $A\in \mathfrak{so}(n+1)$, the vector field
\[
\xi(\bm{x}) = A\bm{x} + \bm{b} - \langle \bm{b}, \bm{x} \rangle \bm{x}
\]
defines a conformal vector field on $\mathbb{S}^n$.
The potential function is given by $\rho(\bm{x})=\langle \bm{b}, \bm{x}\rangle$.
\end{ex}

\begin{ex}
In the real hyperbolic space $\mathbb{R}H^n(-1)$ of constant sectional curvature $-1$, Busemann functions provide a concrete method for constructing conformal vector fields. Given a Busemann function $b$ associated with a boundary point, the gradient of its exponential, $\nabla e^b$, defines a conformal vector field (see \cite{HSHS}).
In the upper half-space model $\{ (\bm{x}, y) \mid \bm{x} \in \mathbb{R}^{n-1},\, y > 0 \}$, a general conformal vector field $\xi$ can be expressed explicitly as
\begin{multline*}
\xi(\bm{x},y)
=\left(\begin{array}{c}\bm{a}_0\\ b_0\end{array}\right)
+\left(\begin{array}{cc}A&-\bm{b}_1\\{ }^t\bm{b}_1&0\end{array}\right)\left(\begin{array}{c}\bm{x}\\ y\end{array}\right)
+a_1 \left(\begin{array}{c}\bm{x}\\ y\end{array}\right)
\\
+(\|\bm{x}\|^2+y^2) \left(\begin{array}{c}\bm{a}_2\\ b_2\end{array}\right)
-2\left\langle \left(\begin{array}{c}\bm{a}_2\\ b_2\end{array}\right),
\left(\begin{array}{c}\bm{x}\\ y\end{array}\right)\right\rangle
\left(\begin{array}{c}\bm{x}\\ y\end{array}\right),
\end{multline*}
where $a_1, b_0, b_2\in\mathbb{R}$, $\bm{a}_0, \bm{a}_2, \bm{b}_1\in \mathbb{R}^{n-1}$ and $A\in \mathfrak{so}(n-1)$.
The potential function is given by $\rho(\bm{x},y)=-\dfrac{1}{y}\left(b_0+\langle \bm{x}, \bm{b}_1\rangle +b_2(\|\bm{x}\|^2+y^2)\right)$.
It is not very difficult to see that this is a conformal vector field. For the derivation, see \cite{HSHS2}.
\end{ex}

\begin{rem}
On the sphere $\mathbb{S}^n(1)$ every non-trivial conformal vector field is non-homothetic but nevertheless complete, while on the Euclidean space $\mathbb{E}^{n}$ the classical dilations provide a complete homothetic (but non-isometric) conformal vector field.
In contrast, on the real hyperbolic space $\mathbb{R}H^{n}(-1)$ every
non-Killing conformal vector field is incomplete:
its integral curve reaches the ideal boundary in finite time and therefore fails to generate a one–parameter subgroup of $\mathrm{Conf}(\mathbb{R}H^{n}(-1))$.
\end{rem}



\subsection{Damek--Ricci spaces}

Let $\mathfrak{n} = \mathfrak{v} \oplus \mathfrak{z}$ be a two-step nilpotent Lie algebra equipped with an inner product $\langle \cdot, \cdot \rangle$, where $\mathfrak{z}$ is the center and $\mathfrak{v}$ its orthogonal complement.  
Since $\mathfrak{z}$ is center, the Lie bracket satisfies  
\[
[\mathfrak{v}, \mathfrak{v}] \subset \mathfrak{z}, \quad [\mathfrak{v}, \mathfrak{z}] = 0.
\]

\begin{dfn}
Such a Lie algebra $\mathfrak{n}$ is called a \emph{generalized Heisenberg algebra} if there exists a linear map $J_Z : \mathfrak{v} \to \mathfrak{v}$ for each $Z \in \mathfrak{z}$, defined by
\[
\langle J_Z V, W \rangle = \langle Z, [V, W] \rangle \quad (V, W \in \mathfrak{v}),
\]
and the map $J : \mathfrak{z} \to \mathrm{End}(\mathfrak{v})$ satisfies the condition
\[
J_Z^2 = -\|Z\|^2 \cdot \mathrm{Id}_{\mathfrak{v}} \quad \text{for all } Z \in \mathfrak{z}.
\]
\end{dfn}

Let $V \in \mathfrak{v}$ be a fixed unit vector. Then the subspace  
$J_{\mathfrak{z}} V := \{ J_Z V \mid Z \in \mathfrak{z} \}$  
is contained in $\mathfrak{v}$.
One has the orthogonal decomposition
\begin{equation}\label{decompv}
\mathfrak{v} = J_{\mathfrak{z}} V \oplus \ker \mathrm{ad}(V),
\end{equation}
where $\ker \mathrm{ad}(V) := \{ W \in \mathfrak{v} \mid [V, W] = 0 \}$ is the kernel of the adjoint map $\mathrm{ad}(V) : \mathfrak{v} \to \mathfrak{z}$.  
This decomposition will play a crucial role in isolating the components of $\mathfrak{v}$ that interact nontrivially with $V$, particularly in the study of conformal vector fields.

Given such a generalized Heisenberg algebra $\mathfrak{n}$, we define a solvable Lie algebra $\mathfrak{s} = \mathfrak{n} \oplus \mathfrak{a}$, where $\mathfrak{a} = \mathbb{R} A$ is one-dimensional, with the bracket relations
\[
[A, V] = \frac{1}{2} V, \quad [A, Z] = Z \quad (V \in \mathfrak{v},\; Z \in \mathfrak{z}).
\]
We equip $\mathfrak{s}$ with an inner product such that the decomposition $\mathfrak{s} = \mathfrak{v} \oplus \mathfrak{z} \oplus \mathfrak{a}$ is orthogonal, the restriction to $\mathfrak{n}$ agrees with the given inner product on $\mathfrak{n}$, and the generator $A$ is a unit vector.

Let $S$ be the simply connected Lie group with Lie algebra $\mathfrak{s}$, equipped with the left-invariant Riemannian metric $g$ induced by the inner product on $\mathfrak{s}$.

\begin{dfn}
Such a Lie group $S$ is called a \emph{Damek--Ricci space}.
\end{dfn}

\begin{rem}
A Riemannian manifold is called \emph{harmonic} if its volume density in geodesic normal coordinates centered at any point depends only on the radius.
Damek--Ricci spaces were introduced by Damek and Ricci \cite{DR} as a family of harmonic manifolds that are not symmetric, thus providing counterexamples to the Lichnerowicz conjecture in the noncompact case.
Here the Lichnerowicz asserted that any harmonic manifold must be locally symmetric.
While they include all noncompact rank-one symmetric spaces, the real hyperbolic space appears as a special case where the center vanishes, and most Damek–Ricci spaces are non-symmetric.
Each Damek--Ricci space is an Einstein manifold with nonpositive sectional curvature.
\end{rem}

By identifying the Damek--Ricci space $S$ with its Lie algebra $\mathfrak{s}$ via the exponential map, we may endow $S$ with global coordinates determined by a fixed orthonormal basis of $\mathfrak{s}$.
Let us fix an orthonormal basis $\{V_1, \ldots, V_k, Z_1, \ldots, Z_m, A\}$ of $\mathfrak{s} = \mathfrak{v} \oplus \mathfrak{z} \oplus \mathfrak{a}$, where each $V_i \in \mathfrak{v}$, $Z_r \in \mathfrak{z}$, and $A$ generates $\mathfrak{a}$.
Then, a natural global coordinate system on $S$ is given by
$$
(v_1, \ldots, v_k, z_1, \ldots, z_m, a),
$$
where $v_i$, $z_r$, and $a$ are the coordinate functions corresponding to $V_i$, $Z_r$, and $A$, respectively.
In these coordinates, the left-invariant vector fields corresponding to the basis can be expressed explicitly as differential operators involving these coordinate functions as follows:
\begin{align}\label{coordS}
V_i=&e^{a/2}\dfrac{\partial}{\partial v_i}
-\dfrac{e^{a/2}}{2}\sum_{r=1}^m\sum_{j=1}^kA^r_{ij}v_j\dfrac{\partial}{\partial z_r},\\
Z_r=&e^{a}\dfrac{\partial}{\partial z_r},\qquad
A=\dfrac{\partial}{\partial a},\notag
\end{align}
where $A^r_{ij}=g([V_i, V_j], Z_r)$.
See \cite[pp.82]{BTV} for detail.

\begin{rem}
When expressing the orthonormal basis vectors $V_i \in \mathfrak{s}$ as left-invariant vector fields on $S$, one must be careful to distinguish the roles of these symbols.
On the left-hand of \eqref{coordS},
$V_i$ denotes a left-invariant vector field on $S$.
On the right-hand side, in the expressions involving the structure constants $A_{ij}^r$,
the $V_i$ and $Z_r$ refer to elements of the Lie algebra $\mathfrak{s}$,
viewed as an abstract vector space equipped with a Lie bracket.
\end{rem}

\begin{rem}\label{exprsV_i}
In expressing the left-invariant vector fields $V_i$ on the Damek--Ricci space $S$,
we note the following:
\begin{enumerate}
\item In the case where $S = \mathbb{C}H^n$,
the center $\mathfrak{z}$ has dimension 1, and $\mathfrak{v}$ has dimension $2(n-1)$.
In this case, one can choose a basis of $\mathfrak{v}$ of the form
$$V_1,\ V_2 = J_Z V_1,\ V_3,\ V_4 = J_Z V_3,\ldots, V_{2n-3},\ V_{2(n-1)} = J_Z V_{2n-3}.$$
Then the expressions for $V_i$ as left-invariant vector fields reduce to
$$
V_{2i-1}=e^{a/2}\left(\dfrac{\partial}{\partial v_{2i-1}}
-\dfrac{v_{2i}}{2}\dfrac{\partial}{\partial z}\right),\quad
V_{2i}=e^{a/2}\left(\dfrac{\partial}{\partial v_{2i}}
+\dfrac{v_{2i-1}}{2}\dfrac{\partial}{\partial z}\right).
$$
\item In the general case, suppose the center $\mathfrak{z}$ has dimension $m$.
If we choose a basis of $\mathfrak{v}$ as 
$$V_1,\ V_2 = J_{Z_1} V_1,\ldots, V_{m+1} = J_{Z_m} V_1,\ V_{m+2}, \ldots, V_k,$$
then $V_{m+2}, \ldots, V_k \in \ker\mathrm{ad}(V_1)$.
In this setting, the expressions for $V_1$ and $V_{i+1}$ $(i=1,\ldots, m)$ simplify accordingly
\begin{align*}
V_1=&e^{a/2}\left(\dfrac{\partial}{\partial v_1}
-\dfrac{v_{2}}{2} \dfrac{\partial}{\partial z_1}
-\dfrac{1}{2}\sum_{j=2}^m v_{j+1}\dfrac{\partial}{\partial z_j}\right),\\
V_{i+1}=&e^{a/2}\left(\dfrac{\partial}{\partial v_{i+1}}+\dfrac{v_1}{2}\dfrac{\partial}{\partial z_i}
-\dfrac{1}{2}\sum_{\substack{r=1\\ r\ne i}}^m\sum_{\substack{j=2\\ j\ne i+1}}^kA^r_{i+1\,j}v_j\dfrac{\partial}{\partial z_r}\right).
\end{align*}
These coordinate expressions will prove useful when analyzing conformal vector fields explicitly in later sections.
\end{enumerate}
\end{rem}

Damek--Ricci spaces include, as special cases, rank-one symmetric spaces of non-compact type other than real hyperbolic spaces.
A Damek--Ricci space $S$ is said to satisfy the \emph{$J^2$-condition} if for any orthogonal vectors $Z_1, Z_2 \in \mathfrak{z}$, there exists $Z_3 \in \mathfrak{z}$ such that $J_{Z_1} \circ J_{Z_2} = J_{Z_3}$.
It is known that a Damek--Ricci space satisfies the $J^2$-condition if and only if it is a symmetric space.
This occurs precisely when $\dim \mathfrak{z} = 1, 3$, or $7$, the corresponding Damek--Ricci space is isometric to the complex hyperbolic space $\mathbb{C}H^n$, quaternionic hyperbolic space $\mathbb{H}H^n$, or the Octonionic hyperbolic plane $\mathbb{O}H^2$, respectively.
However, not all Damek--Ricci spaces with these center dimensions are symmetric; there also exist non-symmetric harmonic spaces in these categories.

\section{Conformal Killing equations on Damek--Ricci spaces}

We now investigate the existence of conformal vector fields on a general Damek--Ricci space $(S, g)$.
As established in the previous subsection, the Lie algebra $\mathfrak{s}$ of $S$ admits an orthogonal decomposition $\mathfrak{s} = \mathfrak{v} \oplus \mathfrak{z} \oplus \mathfrak{a}$,
where $\mathfrak{a} = \mathbb{R}A$ and $\mathfrak{n} := \mathfrak{v} \oplus \mathfrak{z}$ is a generalized Heisenberg algebra.
Let $\{V_1, \ldots, V_k\}$ and $\{Z_1, \ldots, Z_m\}$ be orthonormal bases of $\mathfrak{v}$ and $\mathfrak{z}$, respectively, and denote by the same symbols the associated left-invariant vector fields on $S$.
We consider vector fields on $S$ of the form
$$
\xi = \sum_{i=1}^{k} f_{1i} V_i + \sum_{r=1}^{m} f_{3r} Z_r + f_4 A,
$$
and seek conditions under which such a vector field is conformal with respect to the left-invariant metric $g$.
From Lemma~\ref{property_L}, the Lie derivative of the metric $g$ with respect to a vector field $\xi$ can be written as
\begin{multline}\label{Lxg}
\mathcal{L}_\xi g
= \sum_{i=1}^k f_{1i} \mathcal{L}_{V_i} g + \sum_{r=1}^m f_{3r} \mathcal{L}_{Z_r} g + f_4 \mathcal{L}_A g \\
+ 2 \sum_{i=1}^k \mathrm{sym}(df_{1i} \otimes V_i^\flat)
+ 2 \sum_{r=1}^m \mathrm{sym}(df_{3r} \otimes Z_r^\flat)
+ 2\, \mathrm{sym}(df_4 \otimes A^\flat),
\end{multline}
where $V_i^\flat$, $Z_r^\flat$, and $A^\flat$ denote the dual 1-forms with respect to $g$.

For convenience, we fix a unit vector $V := V_1 \in \mathfrak{v}$ and $Z := Z_1 \in \mathfrak{z}$, and define $V_2 := J_Z V$.
Accordingly, we denote the coordinate functions associated with $v_1$, $v_2$, and $z_1$ by $x$, $y$, and $z$, respectively, as introduced in the previous subsection.
Using the coordinate expressions of the left-invariant frame (see Remark~\ref{exprsV_i}~(2)), the vector fields $V$ and $J_Z V$ are given by
\begin{align}
V &= e^{a/2} \left(
\dfrac{\partial}{\partial x} - \dfrac{y}{2} \dfrac{\partial}{\partial z}
- \dfrac{1}{2} \sum_{r=2}^m v_{r+1} \dfrac{\partial}{\partial z_r}
\right), \\
J_Z V &= e^{a/2} \left(
\dfrac{\partial}{\partial y} + \dfrac{x}{2} \dfrac{\partial}{\partial z}
- \dfrac{1}{2} \sum_{j=2}^m \sum_{j=3}^k A^r_{2j} v_j \dfrac{\partial}{\partial z_r}
\right).
\end{align}
Let us now regard the vector fields as differential operators, and define
$$
D_x:=\dfrac{\partial}{\partial x} - \dfrac{1}{2} \sum_{r=2}^m v_{r+1} \dfrac{\partial}{\partial z_r},\qquad
D_y:=\dfrac{\partial}{\partial y} - \dfrac{1}{2} \sum_{r=2}^m \sum_{j=3}^k A^r_{2j} v_j \dfrac{\partial}{\partial z_r}.
$$
Then the vector fields $V$ and $J_Z V$ can be written as
$$
V = e^{a/2} \left(D_x-\dfrac{y}{2}\dfrac{\partial}{\partial z}\right),\qquad
J_Z V = e^{a/2} \left(D_y+\dfrac{x}{2}\dfrac{\partial}{\partial z}\right).
$$
We remark that the differential operators $D_x$ and $D_y$ are independent of the variables $z$ and $a$; that is, they involve no derivatives with respect to $z$ or $a$, and their coefficients do not depend on $z$ or $a$.

We also abbreviate $f_{11} = f_1$, $f_{12} = f_2$, and $f_{31} = f_3$ for simplicity.
Define the subspace
\[
\mathfrak{s}_0 := \mathrm{span}\{ V, J_Z V, Z, A \}.
\]
Then $\mathfrak{s}_0$ is a Lie subalgebra of $\mathfrak{s}$ and is isomorphic to the Lie algebra of the complex hyperbolic plane.
Let $\mathfrak{v}_0$ denote the orthogonal complement of $\mathrm{span}\{ V, J_Z V \}$ in $\mathfrak{v}$,
and let $\mathfrak{z}_0$ denote the orthogonal complement of $\mathbb{R} Z$ in $\mathfrak{z}$.

Accordingly, the Lie derivative $\mathcal{L}_\xi g$ with respect to vectors in $\mathfrak{s}_0$ can be computed in terms of the components of $\xi$,
and takes the following form:
\begin{equation*}
\mathcal{L}_Vg(V,A)=\frac{1}{2},\quad
\mathcal{L}_Vg(J_ZV,Z)=-1,
\end{equation*}
\begin{equation*}
\mathcal{L}_{J_ZV}g(V,Z)=1,\quad
\mathcal{L}_{J_ZV}g(J_ZV,A)=\dfrac{1}{2},
\end{equation*}
\begin{equation*}
\mathcal{L}_{Z}g(Z,A)=1,
\end{equation*}
\begin{equation*}
\mathcal{L}_{A}g(V,V)=\mathcal{L}_{A}g(J_ZV,J_ZV)=-1,\quad
\mathcal{L}_{A}g(Z,Z)=-2.
\end{equation*}
Therefore, the Lie derivative $\mathcal{L}_V g$ is given by
\begin{align*}\label{Lie_g}
\mathcal{L}_Vg
=&\left(
\begin{array}{cc:cc|c:c}
0&0&0&\frac{1}{2}&\ast&\ast\\
0&0&-1&0&\ast&\ast\\
\hdashline
0&-1&0&0&\ast&\ast\\
\frac{1}{2}&0&0&0&\ast&\ast\smallskip\\
\hline
\ast&\ast&\ast&\ast&\ast&\ast\\
\hdashline
\ast&\ast&\ast&\ast&\ast&\ast
\end{array}
\right)
\begin{array}{c}
\mbox{\tiny $V_1$}\\
\mbox{\tiny $V_2$}\\
\mbox{\tiny $Z$}\\
\mbox{\tiny $A$}\\
\mbox{\tiny $\mathfrak{v}_0$}\\
\mbox{\tiny $\mathfrak{z}_0$}
\end{array}
\end{align*}
Here, we display the Lie derivative $\mathcal{L}_V g$ in matrix form with respect to the orthonormal basis $\{V_1=V, V_2=J_ZV, Z, A, \ldots\}$,
highlighting the structure restricted to the subalgebra $\mathfrak{s}_0$.
Entries marked with $*$ include terms that are either not computed explicitly or are determined by the symmetry of $\mathcal{L}_V g$.
The Lie derivatives with respect to the remaining vector fields are given by the following matrices.
\begin{equation*}
\mathcal{L}_{J_ZV}g
=\left(
\begin{array}{cc:cc|c:c}
0&0&1&0&\ast&\ast\\
0&0&0&\frac{1}{2}&\ast&\ast\smallskip\\
\hdashline
1&0&0&0&\ast&\ast\\
0&\frac{1}{2}&0&0&\ast&\ast\smallskip\\
\hline
\ast&\ast&\ast&\ast&\ast&\ast\\
\hdashline
\ast&\ast&\ast&\ast&\ast&\ast
\end{array}
\right),\quad
\mathcal{L}_Zg
=\left(
\begin{array}{cc:cc|c:c}
0&0&0&0&\ast&\ast\\
0&0&0&0&\ast&\ast\\
\hdashline
0&0&0&1&\ast&\ast\\
0&0&1&0&\ast&\ast\\
\hline
\ast&\ast&\ast&\ast&\ast&\ast\\
\hdashline
\ast&\ast&\ast&\ast&\ast&\ast
\end{array}
\right),
\end{equation*}
\begin{equation*}
\mathcal{L}_Ag
=\left(
\begin{array}{cc:cc|c:c}
-1&0&0&0&\ast&\ast\\
0&-1&0&0&\ast&\ast\\
\hdashline
0&0&-2&0&\ast&\ast\\
0&0&0&0&\ast&\ast\\
\hline
\ast&\ast&\ast&\ast&\ast&\ast\\
\hdashline
\ast&\ast&\ast&\ast&\ast&\ast
\end{array}
\right).
\end{equation*}

Therefore, from \eqref{Lxg}, by further decomposing the $\mathfrak{s}_0$-block part of the equation $\mathcal{L}_\xi g = 2\rho g$ into sub-blocks, we can write it in the following form.\begin{align}
\label{11}
\left(
\begin{array}{cc}
-f_4&0\\
0&-f_4
\end{array}
\right)
+\left(
\begin{array}{cc}
2 V f_1&V f_2 + J_Z V f_1\\
V f_2 + J_Z V f_1&2 J_Z V f_2
\end{array}
\right)
=&\left(
\begin{array}{cc}
2\rho&0\\
0&2\rho
\end{array}
\right),\\
\label{21}
\left(
\begin{array}{cc}
f_2&-f_1\\
\frac{1}{2}f_1&\frac{1}{2}f_2
\end{array}
\right)
+\left(
\begin{array}{cc}
V f_3 + Z f_1&J_Z V f_3 + Z f_2\\
V f_4 + A f_1&J_Z V f_4 + A f_2
\end{array}
\right)
=&\left(
\begin{array}{cc}
0&0\\
0&0
\end{array}
\right),\\
\label{22}
\left(
\begin{array}{cc}
-2f_4&f_3\\
f_3&0
\end{array}
\right)
+\left(
\begin{array}{cc}
2 Z f_3&Z f_4 + A f_3\\
Z f_4 + A f_3&2 A f_4
\end{array}
\right)=&\left(
\begin{array}{cc}
2\rho&0\\
0&2\rho
\end{array}
\right).
\end{align}
From the $(2,2)$ component of equation \eqref{22}, we have
\begin{equation}\label{pf}
\rho = \dfrac{\partial f_4}{\partial a}.
\end{equation}
Therefore, if $f_4$ is independent of the variable $a$,
then $\rho = 0$, and hence $\xi$ is a Killing vector field.
By substituting \eqref{pf} into the remaining equations in \eqref{22} and simplifying,
we obtain
$$
\dfrac{\partial}{\partial z}(e^a f_3)=e^{-a}\dfrac{\partial}{\partial a}(e^a f_4),\qquad
e^{-a}\dfrac{\partial}{\partial a}(e^a f_3)=-\dfrac{\partial}{\partial z}(e^a f_4)
$$
Hence, by introducing the change of variable $w := e^a$ and setting $F_i := e^a f_i$ for $i = 3, 4$,
the two equations above state that $F_3$ and $F_4$ satisfy the Cauchy--Riemann equations, i.e.,
the complex function $F_3 + i F_4$ is holomorphic on the $(z, w)$-plane.
It follows that $F_3$ and $F_4$ are harmonic functions on the $(z, w)$-plane,
parametrized by the remaining variables $x, y, v_3, \ldots, v_k$, and $z_2, \ldots, z_m$.

\section{Holomorphic and harmonic structure of coefficients of $Z$ and $A$}

For each fixed parameters $\bm{n}_0:=(x, y, v_3, \ldots, v_k, z_2, \ldots, z_m)$,
the map $(z,w)\mapsto (F_3, F_4)$ is harmonic,
so every point $(z_0,w _0)\in \mathbb{R}\times\mathbb{R}_+$ admits a neighbourhood in which $F_3$ and $F_4$ possess the expansion in homogeneous harmonic polynomials,
for example around the origin $(0, 0)$ we have
\begin{align}
F_3&=\sum_{m=0}^{\infty}
\bigl\{A_m\,\Re[(z+iw)^m]-B_m\,\Im[(z+iw)^m]\bigr\},\\
F_4&=\sum_{m=0}^{\infty}
\bigl\{A_m\,\Im[(z+iw)^m]+B_m\,\Re[(z+iw)^m]\bigr\},
\end{align}
where the coefficient functions $A_m, B_m$ depend smoothly on $(x, y, \bm{v}_0, \bm{z}_0)$ and the series converges uniformly on compact subsets.
This follows from the completeness of homogeneous harmonic polynomials (see Proposition 5.5 and Theorem 5.25 of \cite{ABR}); no logarithmic or negative‐power terms appear because we require $F_3$ and $F_4$ to be smooth on the whole manifold.

\begin{rem}
A non-constant function depending only on the radius $r=\sqrt{z^{2}+w^{2}}$ is harmonic in two variables if and only if it equals $c_1\log r+c_0$, which is singular at $r=0$.
Hence such radial terms do not define a smooth function on any domain containing $r=0$; in particular, they cannot occur in the local analysis of smooth conformal vector fields near points where $r=0$.
We therefore exclude them and assume $c_1=0$, so that the constant ($m=0$) term is the only radial contribution allowed.
\end{rem}

For each $m\ge0$, we set functions $F_3^{[m]}$ and $F_4^{[m]}$ as
$$
(C_1^{[m]} + i C_2^{[m]}) (z + i w)^m=F_3^{[m]} + i F_4^{[m]},
$$
i.e., 
\begin{align*}
F_3^{[m]}= &\Re\left[ (C_1^{[m]} + i C_2^{[m]}) (z + i w)^m\right],\\
F_4^{[m]}= &\Im\left[ (C_1^{[m]} + i C_2^{[m]}) (z + i w)^m\right],
\end{align*}
where $C_i^{[m]}$, $i=1, 2$, are smooth functions depend smoothly on $\bm{n}_0$.
This decomposition allows us to systematically analyze the real and imaginary parts of the function through harmonic polynomial expansions.
A straightforward manipulation then yields
\begin{align}
F_3^{[m]}
= & \frac{1}{2} \left\{ (C_1^{[m]} + i C_2^{[m]}) (z + i w)^m + (C_1^{[m]} - i C_2^{[m]}) (z - i w)^m\right\},\notag\\
=& \frac{1}{2} \left\{ C_1^{[m]}((z + i w)^m+(z - i w)^m)
+i C_2^{[m]}((z + i w)^m-(z - iw)^m)\right\}.\label{f3_o}
\end{align}
Next, we invoke the binomial expansion of $(z \pm i w)^m$,
which can be written as
\begin{equation*}
(z \pm i w)^m
= \sum_{k=0}^{m} \binom{m}{k} z^{m-k} (\pm i w)^k
= \sum_{k=0}^{m} \binom{m}{k} z^{m-k}\,i^{-k}\, w^k,
\end{equation*}
where we have used
$(\pm i)^k=\left(i^{\pm 1}\right)^k=i^{\pm k}$
up to signs that align with the binomial terms.
Each term in this expansion corresponds to a specific harmonic component in the decomposition.
To handle the powers of $i$,
we further utilize the standard form of Euler's formula,
$$
i^\theta=\cos \frac{\pi}{2}\theta + i \sin \frac{\pi}{2}\theta.
$$
Combining these expressions reveals that all relevant real and imaginary contributions stem from the trigonometric parts $\cos\left(\frac{\pi}{2}k\right)$ and $\sin\left(\frac{\pi}{2}k\right)$.
Consequently, upon collecting real terms, one obtains
\begin{equation*}
F_3^{[m]}=\sum_{k=0}^{m} \binom{m}{k} z^{m-k} w^k \left(C_1^{[m]} \cos \frac{\pi}{2}k - C_2^{[m]} \sin \frac{\pi}{2}k \right).
\end{equation*}
Hence, $F_3^{[m]}$ can be expressed as a finite sum of monomials in $z$ and $w$,
each weighted by trigonometric coefficients determined by $\cos\left(\frac{\pi}{2}k\right)$ and $\sin\left(\frac{\pi}{2}k\right)$.
This decomposition makes explicit the real-part structure inherent in $(z+iw)^m$ under the given linear combination involving $C_1^{[m]}$ and $C_2^{[m]}$.
Likewise, $F_4^{[m]}$ takes the analogous form given by
\begin{equation*}
F_4^{[m]}
=\sum_{k=0}^{m} \binom{m}{k} z^{m-k} w^k \left(C_1^{[m]} \sin \frac{\pi}{2}k + C_2^{[m]} \cos \frac{\pi}{2}k \right).
\end{equation*}

Using $F_{3}^{[m]}$ and $F_{4}^{[m]}$,
the functions $f_{3}$ and $f_{4}$ can be represented by the following infinite series.
\begin{align}
f_3= &e^{-a}\left(\sum_{m=1}^{\infty}F_{3}^{[m]}+ C_3\right)\notag\\
=& e^{-a}\left\{\sum_{m=1}^{\infty}\left(\sum_{k=0}^{m} \binom{m}{k} z^{m-k} e^{k a} \left(C_1^{[m]} \cos \frac{\pi}{2}k - C_2^{[m]} \sin \frac{\pi}{2}k \right)\right) + C_3\right\},\label{f3simple}\\
f_4 = & e^{-a}\left(\sum_{m=1}^\infty F_4^{[m]}+ C_4\right)\notag\\
= & e^{-a}\left\{\sum_{m=1}^\infty\left(\sum_{k=0}^{m} \binom{m}{k} z^{m-k} e^{k a} \left(C_1^{[m]} \sin \frac{\pi}{2}k + C_2^{[m]} \cos \frac{\pi}{2}k \right)\right) + C_4\right\},\label{f4simple}
\end{align}
where $C_3$ and $C_4$ are some functions on $\bm{n}_0\in \mathbb{R}^{k+m-1}$.

Equations \eqref{11} and \eqref{21} imply
\begin{align}
\label{11-(1,1)}
\left(D_x-\dfrac{y}{2}\dfrac{\partial}{\partial z}\right)f_1=&e^{-a}\dfrac{\partial}{\partial a}\left(e^{a/2}f_4\right),\\
\label{11-(2,1)}
\left(D_y+\dfrac{x}{2}\dfrac{\partial}{\partial z}\right)f_1+\left(D_x-\dfrac{y}{2}\dfrac{\partial}{\partial z}\right)f_2=&0,\\
\label{11-(2,2)}
\left(D_y+\dfrac{x}{2}\dfrac{\partial}{\partial z}\right)f_2=&e^{-a}\dfrac{\partial}{\partial a}\left(e^{a/2}f_4\right),\\
\label{12-(1,1)}
e^{a} \frac{\partial f_1}{\partial z}+e^{a/2}\left(D_x-\dfrac{y}{2}\dfrac{\partial}{\partial z}\right)f_3+f_2=&0,\\
\label{12-(1,2)}
e^{a} \frac{\partial f_2}{\partial z}+e^{a/2}\left(D_y+\dfrac{x}{2}\dfrac{\partial}{\partial z}\right)f_3-f_1=&0,\\
\label{12-(2,1)}
e^a\left(D_x-\dfrac{y}{2}\dfrac{\partial}{\partial z}\right)f_4+\dfrac{\partial}{\partial a}\left(e^{a/2}f_1\right)
=&0,\\
\label{12-(2,2)}
e^a\left(D_y+\dfrac{x}{2}\dfrac{\partial}{\partial z}\right)f_4+\dfrac{\partial}{\partial a}\left(e^{a/2}f_2\right)
=&0.
\end{align}
Thus, in order to analyze the conformal vector field condition \eqref{CVF},
we focus on the $\mathfrak{s}_0$-block of the equation $\mathcal{L}_\xi g = 2\rho g$.
This leads to a subsystem of partial differential equations \eqref{11-(1,1)} through \eqref{12-(2,2)} involving the functions $f_1, \ldots, f_4$, where $f_3$ and $f_4$ are given explicitly by \eqref{f3simple} and \eqref{f4simple}.
Although this subsystem does not capture the full conformal Killing equation on $S$, it provides sufficient constraints to determine the vanishing of the potential function $\rho$.

\begin{rem}
It should be noted that (\ref{f3simple}) and (\ref{f4simple}) they represent power series in the neighbourhood of zero, consequently the solution pair $f_3$ and $f_4$ is in a neighbourhood of zero.
The solution in any neighbourhood around  some point $\alpha + i \beta $ is given by replacing $z + i w$ by $(z-\alpha) + i (w - \beta)$. Hence, the above computations remains unaltered.
\end{rem}

\section{Main result and its proof}\label{mrp}

We now state our main result.

\begin{thm}\label{mainresult}
Under the assumptions given above, $C_{4} = 0$.
Moreover, for every positive integer $m$,
$C_{1}^{[m]} = 0$, and for all positive integer $m \ge 3$, we have $C_{2}^{[m]} = 0$.
\end{thm}

\begin{proof}
To simplify the notation in what follows, we define
\[
\mathcal{C}^{(p)}_{[m,k]}:=
D_x^p C_1^{[m]}\sin\dfrac{\pi}{2} k+D_x^p C_2^{[m]}\cos\dfrac{\pi}{2} k,
\]
for integers $m \ge 1$, $k \ge 0$, and $p = 0, 1, 2$.

From equations \eqref{12-(2,1)} and \eqref{12-(2,2)}, $f_1$ and $f_2$ can be determined in terms of $f_4$.
For instance, $f_1$ can be expressed by the following integral formula:
\begin{equation}\label{f1_0}
e^{a/2}f_1=-\int e^a\left(D_x f_4-\dfrac{y}{2}\dfrac{\partial f_4}{\partial z}\right)\,da+C_5(\bm{n}_0,z),
\end{equation}
where $C_5(\bm{n}_0,z)$ is an integration ``constant'' that may depend on $(\bm{n}_0,z)$ but not on $a$.
Equation \eqref{f1_0} could be integrated further to obtain an explicit expression for $f_1$.
However, we will not pursue that computation here; instead, we proceed with the proof of our main result.

From \eqref{11-(1,1)}, we have
\begin{align}\label{16}
0=& - D_x\left(e^{a/2} f_1\right)
+\dfrac{y}{2}\dfrac{\partial}{\partial z}\left(e^{a/2} f_1\right)
+e^{-a/2}\dfrac{\partial}{\partial a}\left(e^{a/2} f_4\right).
\end{align}
The first and second terms on the right-hand side of \eqref{16} become
$$
\int e^a\left(
D_x^2 f_4
-y\dfrac{\partial D_x f_4}{\partial z}
+\dfrac{y^2}{4}\dfrac{\partial^2 f_4}{\partial z^2}
\right) da
-D_x C_5
+\dfrac{y}{2} \dfrac{\partial C_5}{\partial z}.
$$
By computing each term of the above equation, we obtain the following:
\begin{align*}
\int e^{a}D_x^2 f_4 da
=&\int \left\{\sum_{m=1}^\infty\left(\sum_{k=0}^{m} \binom{m}{k} z^{m-k} e^{k a} \mathcal{C}^{(2)}_{[m,k]} + D_x^2C_4\right)\right\} da\\
=&\sum_{m=1}^\infty\left(\sum_{k=1}^{m} \binom{m}{k}\dfrac{1}{k} z^{m-k} e^{k a} \mathcal{C}^{(2)}_{[m,k]}\right)
+ a \left(
\sum_{m=1}^\infty z^m D_x^2 C^{[m]}_2+ D_x^2 C_4
\right).
\end{align*}

Now, after some computations and simplifications, we affirm
\begin{align*}
\int e^{a}\dfrac{\partial D_x f_4}{\partial z} da
=&\dfrac{\partial}{\partial z}\int \left\{\sum_{m=1}^\infty\left(\sum_{k=0}^{m} \binom{m}{k} z^{m-k} e^{k a} \mathcal{C}^{(1)}_{[m,k]}\right) + D_x C_4\right\} da\\
=&\sum_{m=1}^\infty\left(\sum_{k=1}^{m} \binom{m+1}{k}\dfrac{m-k+1}{k} z^{m-k} e^{k a} \mathcal{C}^{(1)}_{[m+1,k]}\right)
+a \sum_{m=1}^\infty m z^{m-1} D_x C_2^{[m]}.
\end{align*}
And,
\begin{align*}
\int e^{a}\dfrac{\partial^2 f_4}{\partial z^2} da
=&\dfrac{\partial^2}{\partial z^2}
\int
\left\{\sum_{m=1}^\infty\left(\sum_{k=0}^{m} \binom{m}{k} z^{m-k} e^{k a} \mathcal{C}^{(0)}_{[m,k]}\right)\right\}
da\\
=& a \sum_{m=1}^\infty (m+1)m z^{m-1} C_2^{[m+1]}\\
&+\sum_{m=1}^\infty\left(\sum_{k=1}^{m} \binom{m+2}{k} \dfrac{(m-k+2)(m-k+1)}{k}z^{m-k} e^{k a} \mathcal{C}^{(0)}_{[m+2,k]}\right).
\end{align*}
Furthermore, the third term of \eqref{16} is given by the following.
\begin{align*}
e^{-a/2}\dfrac{\partial}{\partial a}\left(e^{a/2} f_4\right)
=& \sum_{m=1}^\infty\left(\sum_{k=0}^{m} \binom{m}{k} \left(k-\dfrac{1}{2}\right)z^{m-k} e^{(k - 1)a} \mathcal{C}^{(0)}_{[m,k]}\right)-\dfrac{1}{2}e^{-a} C_4\\
=&\sum_{m=1}^\infty\left(\sum_{k=1}^{m} \binom{m+1}{k+1} \left(k+\dfrac{1}{2}\right)z^{m-k} e^{k a} \mathcal{C}^{(0)}_{[m+1,k+1]}\right)\\
&+\dfrac{1}{2}\sum_{m=1}^\infty m z^{m-1} C^{[m]}_1
-\dfrac{1}{2} e^{-a}\sum_{m=1}^\infty z^m C^{[m]}_2
-\dfrac{1}{2}e^{-a} C_4.
\end{align*}
Thus, equation \eqref{16} can be rewritten as follows:
\begin{align*}
0=&
\sum_{m=1}^\infty\left(\sum_{k=1}^{m} \binom{m}{k}\dfrac{1}{k} z^{m-k} e^{k a} \mathcal{C}^{(2)}_{[m,k]}\right)
+ a \left(
\sum_{m=1}^\infty z^m D_x^2 C^{[m]}_2
+ D_x^2 C_4
\right)\\
&-y\sum_{m=1}^\infty\left(\sum_{k=1}^{m} \binom{m+1}{k}\dfrac{m-k+1}{k} z^{m-k} e^{k a} \mathcal{C}^{(1)}_{[m+1,k]}\right)\\
&-ay \sum_{m=1}^\infty m z^{m-1} D_x C_2^{[m]}
+\dfrac{ay^2}{4} \sum_{m=1}^\infty (m+1)m z^{m-1} C_2^{[m+1]}\\
&+ \dfrac{y^2}{4}\sum_{m=1}^\infty\left(\sum_{k=1}^{m} \binom{m+2}{k} \dfrac{(m-k+2)(m-k+1)}{k}z^{m-k} e^{k a} \mathcal{C}^{(0)}_{[m+2,k]}\right)
-D_x C_5
+\dfrac{y}{2} \dfrac{\partial C_5}{\partial z}\\
&+\sum_{m=1}^\infty\left(\sum_{k=1}^{m} \binom{m+1}{k+1} \left(k+\dfrac{1}{2}\right)z^{m-k} e^{k a} \mathcal{C}^{(0)}_{[m+1,k+1]}\right)\\
&+\dfrac{1}{2}\sum_{m=1}^\infty m z^{m-1} C^{[m]}_1
-\dfrac{1}{2} e^{-a}\sum_{m=1}^\infty z^m C^{[m]}_2
-\dfrac{1}{2}e^{-a} C_4.
\end{align*}
That is,
\begin{align*}
0=
&\sum_{m=1}^\infty \sum_{k=1}^{m} z^{m-k} e^{k a}
\left\{
\binom{m}{k}\dfrac{1}{k} \mathcal{C}^{(2)}_{[m,k]}
-y \binom{m+1}{k}\dfrac{m-k+1}{k} \mathcal{C}^{(1)}_{[m+1,k]}\right.\\
&\left.+ \dfrac{y^2}{4} \binom{m+2}{k} \dfrac{(m-k+2)(m-k+1)}{k} \mathcal{C}^{(0)}_{[m+2,k]}
+ \binom{m+1}{k+1} \left(k+\dfrac{1}{2}\right) \mathcal{C}^{(0)}_{[m+1,k+1]}\right\}\\
&+ a \left\{
\sum_{m=1}^\infty z^m D_x^2 C^{[m]}_2
+ D_x^2 C_4
+\sum_{m=1}^\infty m z^{m-1}\left(-y D_x C_2^{[m]}
+ \dfrac{m+1}{4} y^2 C_2^{[m+1]}\right)
\right\}\\
&-\dfrac{1}{2}e^{-a}\left(
\sum_{m=1}^\infty z^m C^{[m]}_2+C_4
\right)
+\dfrac{1}{2}\sum_{m=1}^\infty m z^{m-1} C^{[m]}_1
- D_x C_5
+\dfrac{y}{2} \dfrac{\partial C_5}{\partial z}.
\end{align*}
Since the above equation holds for any $a$,
we confirm
\begin{multline}\label{16mk}
\sum_{m=k}^\infty z^{m-k}
\left\{
\binom{m}{k}\dfrac{1}{k} \mathcal{C}^{(2)}_{[m,k]}
-y \binom{m+1}{k}\dfrac{m-k+1}{k} \mathcal{C}^{(1)}_{[m+1,k]}\right.\\
+ \dfrac{y^2}{4} \binom{m+2}{k} \dfrac{(m-k+2)(m-k+1)}{k} \mathcal{C}^{(0)}_{[m+2,k]}\\
\left.+ \binom{m+1}{k+1} \left(k+\dfrac{1}{2}\right) \mathcal{C}^{(0)}_{[m+1,k+1]}\right\}=0,
\end{multline}
for any positive integer $k$ and
\begin{equation}
\label{16-1}
\sum_{m=1}^\infty \left\{z^m D_x^2 C^{[m]}_2
+ m z^{m-1}\left(-y D_x C_2^{[m]}
+ \dfrac{m+1}{4} y^2 C_2^{[m+1]}\right)\right\} + D_x^2 C_4 = 0,
\end{equation}
\begin{equation}
\label{16-2}
\sum_{m=1}^\infty z^m C^{[m]}_2 + C_4 = 0,
\end{equation}
\begin{equation}
\label{16-3}
\dfrac{1}{2}\sum_{m=1}^\infty m z^{m-1} C^{[m]}_1
-D_x  C_5
+\dfrac{y}{2} \dfrac{\partial C_5}{\partial z} = 0.
\end{equation}
Since the equation \eqref{16-2} holds for any $z$,
we obtain
$$
C_4(\bm{n}_0)=0,\qquad
C_2^{[m]}(\bm{n}_0)=0\quad( m=1,2,\ldots).
$$
It is immediately clear that these satisfy equation \eqref{16-1} without contradiction.
On the other hand,
since equation \eqref{16mk} holds for any $z$,
it follows that for any positive integer $k$ and for any $m\ge k$,
\begin{multline*}
\binom{m}{k}\dfrac{1}{k} \mathcal{C}^{(2)}_{[m,k]}
-y \binom{m+1}{k}\dfrac{m-k+1}{k} \mathcal{C}^{(1)}_{[m+1,k]}\\
+ \dfrac{y^2}{4} \binom{m+2}{k} \dfrac{(m-k+2)(m-k+1)}{k} \mathcal{C}^{(0)}_{[m+2,k]}\\
+ \binom{m+1}{k+1} \left(k+\dfrac{1}{2}\right) \mathcal{C}^{(0)}_{[m+1,k+1]}=0,
\end{multline*}
holds.
For $k=2$, we obtain
\begin{multline*}
-\dfrac{m-1}{4}D_x^2 C_2^{[m]}
+\dfrac{(m+1)}{4} y D_x C_2^{[m+1]}
-\dfrac{(m+2)(m+1)(m-1)}{16} y^2 C_2^{[m+2]}\\
-(m+1)(m-1)\dfrac{5}{12}C_1^{[m+1]}=0,
\end{multline*}
for any $m\ge 2$.
Given $C_2^{[m]}(\bm{n}_0)=0$ for any positive integer $m$,
it follows that $C_1^{[m]}(\bm{n}_0)=0$ for $m\ge 3$.
\end{proof}

\begin{rem}
For $k=1$, we obtain
\begin{multline}\label{16mk=1}
D_x^2 C_1^{[m]}
-(m+1) y D_x C_1^{[m+1]}
+\dfrac{(m+2)(m+1)}{4} y^2 C_1^{[m+2]}\\
-\dfrac{3(m+1)}{4}C_2^{[m+1]}=0,
\end{multline}
for any $m\ge 1$.
From the previous results, equation \eqref{16mk=1} is meaningful only for $m = 1$ and $m = 2$.
In fact, for $m = 3$, all terms in \eqref{16mk=1} vanish.
For $m=1,2$, equation  \eqref{16mk=1} takes the forms:
\begin{equation}\label{C1[1][2]-1}
D_x^2 C_1^{[1]}
-2 y D_x C_1^{[2]}=0,\qquad
D_x^2 C_1^{[2]}=0.
\end{equation}
Recall that this result was derived from equations \eqref{11-(2,2)} and \eqref{12-(2,1)}.
A similar argument can be applied to equations \eqref{11-(1,1)} and \eqref{12-(2,2)},
leading to the following result, which is analogous to equation \eqref{C1[1][2]-1}:
\begin{equation*}\label{C1[1][2]-2}
D_y^2 C_1^{[1]}
+2 x D_y C_1^{[2]}=0,\qquad
D_y^2 C_1^{[2]}=0,
\end{equation*}
which implies that $C_1^{[1]}$ and $C_1^{[2]}$ are at most first-degree polynomials in $\bm{n}_0$.
\end{rem}

\begin{proof}[Proof of Theorem \ref{mainDR}]
Based on the above considerations, let us explicitly express $f_4$.
From \eqref{f4simple}, we have
\begin{align*}
f_4 = & e^{-a}\left(\sum_{m=1}^2\left(\sum_{k=0}^{m} \binom{m}{k} z^{m-k} e^{k a} C_1^{[m]} \sin \frac{\pi}{2}k\right)\right)\\
=&e^{-a}\left(
\sum_{k=0}^{1} \binom{1}{k} z^{1-k} e^{k a} C_1^{[1]} \sin \frac{\pi}{2}k
+
\sum_{k=0}^{2} \binom{2}{k} z^{2-k} e^{k a} C_1^{[2]} \sin \frac{\pi}{2}k
\right)\\
=&e^{-a}\left(
e^{a} C_1^{[1]}
+
2 z e^{a} C_1^{[2]}
\right)\\
=&C_1^{[1]}(\bm{n}_0)+2 z C_1^{[2]}(\bm{n}_0),
\end{align*}
which shows that $f_4$ does not depend on the variable $a$,
implying that from \eqref{pf} the potential function $\rho$ of $\xi$ is zero.
\end{proof}

\begin{fc}
As noted before, Damek-Ricci space $S$ contains totally geodesic submanifold $\mathbb{C}H^2$ which does not carry any non-homothetic conformal vector fields.

\end{fc}

This suggests the following natural generalization to be explored in future.

\begin{problem}
Let $M$ be a Riemannian manifold containing a totally geodesic submanifold $N$.
Determine what additional assumptions on 
$N$ and on the ambient manifold $M$ are necessary in order for the \emph{infinitesimal conformal rigidity} of $N$ (i.e.\ absence of non–Killing conformal vector fields on $N$) to imply the infinitesimal conformal rigidity of $M$.
\end{problem}

\begin{rem}
Nickerson \cite{Nickerson1985} showed that on a non–conformally flat symmetric space, any conformal change of the metric that preserves local symmetry must be a constant rescaling.
In the class of Damek–Ricci spaces, this observation applies in particular to the rank-one symmetric cases—namely the complex, quaternionic, and Cayley hyperbolic spaces—and implies the nonexistence of nontrivial conformal vector fields on these spaces as well.
For further discussion from this perspective, we refer the reader to \cite{HSHS}.    
\end{rem}

\vspace{0.1in}

\begin{flushleft}
Hiroyasu Satoh\\
Liberal Arts and Sciences\\
Nippon Institute of Technology\\
4-1 Gakuendai, Miyashiro-machi, Saitama 345-8501 Japan.\\
E-mail: hiroyasu@nit.ac.jp
\end{flushleft}

\begin{flushleft}
Hemangi  Madhusudan Shah\\
	Harish-Chandra Research Institute\\ 
	A CI of Homi Bhabha National Institute\\ 
	Chhatnag Road, Jhunsi, Prayagraj-211019, India.\\	
	E-mail:	hemangimshah@hri.res.in
\end{flushleft}

\end{document}